\documentclass[letterpaper,11pt,twoside,keywordsasfootnote,addressatend,noinfoline]{article}
\usepackage{fullpage}
\usepackage[english]{babel}
\usepackage{amssymb}
\usepackage{amsmath}
\usepackage{theorem}
\usepackage{epsfig}
\usepackage{subfigure}
\usepackage{mathrsfs}
\usepackage{color}
\usepackage{imsart}

\newtheorem{theorem}{Theorem}

\newtheorem{lemma}[theorem]{Lemma}

\newenvironment{proof}{\noindent{\bf Proof.}}{\hspace*{2mm}~$\square$}
\newenvironment{proofof}[1]{\noindent{\bf Proof of #1.}}{\hspace*{2mm}~$\square$}

\newcommand{\N}{\mathbb{N}}

\newcommand{\G}{\mathscr G}
\newcommand{\V}{\mathscr V}
\newcommand{\E}{\mathscr E}
\newcommand{\C}{\mathscr C}

\newcommand{\ind}{\mathbf{1}}
\newcommand{\ep}{\epsilon}
\newcommand{\n}{\hspace*{-5pt}}
\newcommand{\unif}{X}
\newcommand{\imex}{Y}
\newcommand{\save}{Z}

\DeclareMathOperator{\card}{card}
\DeclareMathOperator{\uniform}{Uniform \,}

\begin{document}

\begin{frontmatter}
\title     {Rigorous results for the distribution of \\ money on connected graphs}
\runtitle  {distribution of money on connected graphs}
\author    {Nicolas Lanchier\thanks{Research supported in part by NSA Grant MPS-14-040958.} and Stephanie Reed}
\runauthor {Nicolas Lanchier and Stephanie Reed}
\address   {School of Mathematical and Statistical Sciences \\ Arizona State University \\ Tempe, AZ 85287, USA.}

\maketitle

\begin{abstract} \ \
 This paper is concerned with general spatially explicit versions of three stochastic models for the dynamics of money that have been introduced and
 studied numerically by statistical physicists:
 the uniform reshuffling model, the immediate exchange model and the model with saving propensity.
 All three models consist of systems of economical agents that consecutively engage in pairwise monetary transactions.
 Computer simulations performed in the physics literature suggest that, when the number of agents and the average amount of money per agent
 are large, the distribution of money at equilibrium approaches the exponential distribution for the first model, the gamma distribution with shape
 parameter two for the second model and a gamma distribution whose shape parameter depends on the saving propensity for the third model.
 The main objective of this paper is to give rigorous proofs of and extend these conjectures to generalizations of the first two models and a
 variant of the third model that include local rather than global interactions, i.e., instead of choosing the two interacting agents uniformly at
 random from the system, the agents are located on the vertex set of a general connected graph and can only interact with their neighbors.
\end{abstract}



\end{frontmatter}


\section{Introduction}
\label{sec:intro}

\indent The objective of this paper is to give rigorous proofs of various conjectures (as well as extensions of these conjectures) about general spatially
 explicit versions of models in econophysics describing the dynamics of money within a population of economical agents.
 The term econophysics was coined by physicist Eugene Stanley to refer to the subfield of statistical physics that applies concepts from traditional
 physics to economics.
 The terminology is motivated by the idea that molecules can be viewed as individuals, energy as money, and collisions between two molecules as exchanges
 of money between two individuals.
 The models we consider in this paper are known in the mathematics literature as interacting particle systems~\cite{liggett_1985} and are inspired from
 models for the dynamics of money reviewed in~\cite{yakovenko_rosser_2009} that consist of large systems of~$N$ economic agents that interact to engage
 in pairwise monetary transactions.
 The models in~\cite{yakovenko_rosser_2009} are examples of discrete-time Markov chains where, at each time step, two agents are selected uniformly at
 random to interact, which results in an exchange of money between the two agents in an overall conservative system, meaning that the total amount of money
 in the system, say~$M$ dollars, remains constant.
 By analogy with the temperature in physics, the average amount of money per agent~$T = M/N$ is called the money temperature.
 The main problem about these models is to find the distribution of money, i.e., the probability that a given agent has a given amount of money, at
 equilibrium. \\
\indent The first paper introducing such models is~\cite{dragulescu_yakovenko_2000} where several rules for the exchange of money are considered.
 In the most natural version, called the uniform reshuffling model, the total amount of money the two interacting agents possess before the interaction
 is uniformly redistributed between the two agents after the interaction.
 More precisely, using the same notation as in the review~\cite{yakovenko_rosser_2009} and letting~$m_i$ and~$m'_i$ be the amount of money agent~$i$ has
 before and after the interaction, respectively, an interaction between agents~$i$ and~$j$ results in the update
\begin{equation}
\label{eq:reshuffling}
  \begin{array}{rcl}
    m_i & \n \to \n & m'_i = \ep \,(m_i + m_j) \vspace*{2pt} \\
    m_j & \n \to \n & m'_j = (1 - \ep)(m_i + m_j) \end{array} \quad \hbox{where} \quad \ep = \uniform (0, 1).
\end{equation}
 The computer agent-based simulations performed in~\cite{dragulescu_yakovenko_2000} suggest that, for all the rules under consideration
 including~\eqref{eq:reshuffling}, the limiting distribution of money approaches an exponential distribution with mean~$T$ when both the population
 size~$N$ and the money temperature~$T$ are large, i.e., the probability that a given individual has~$m$ dollars at equilibrium approaches
 $$ P (m) = \frac{1}{T} \ e^{- m/T} \quad \hbox{where} \quad T = M/N. $$
 Strictly speaking, since the independent uniform random variables~$\ep$ used at each update are continuous, the probability of having exactly~$m$ dollars
 at equilibrium is equal to zero, so the limit above given in the physics literature has to be understood as follows:
 The probability that a given individual has at least~$m$ dollars at equilibrium approaches
 $$ \int_m^{\infty} \frac{1}{T} \ e^{- x/T} \,dx = e^{-m / T}. $$
 One of the models in~\cite{dragulescu_yakovenko_2000} assumes that one of the two interacting agents chosen at random gives one dollar to the
 other agent if she indeed has at least one dollar.
 The simulations in~\cite{dragulescu_yakovenko_2000} suggest that the limiting distribution of money for this model also approaches the exponential
 distribution, which has been proved analytically and extended to general spatial models in~\cite{lanchier_2017b}. \\
\indent The second model we consider in this paper is inspired from the so-called immediate exchange model introduced and studied numerically
 in~\cite{heinsalu_patriarca_2014}.
 In this model, two agents are again chosen uniformly at random at each time step, but we now assume that each of the two interacting agents gives
 a random fraction of her fortune to the other agent.
 More precisely, an interaction between agents~$i$ and~$j$ results in the update
\begin{equation}
\label{eq:exchange}
  \begin{array}{rcl}
    m_i & \n \to \n & m'_i = (1 - \ep_i) \,m_i + \ep_j \,m_j \vspace*{2pt} \\
    m_j & \n \to \n & m'_j = (1 - \ep_j) \,m_j + \ep_i \,m_i \end{array} \quad \hbox{where} \quad \ep_i, \ep_j = \uniform (0, 1)
\end{equation}
 are independent.
 Note that the uniform reshuffling model~\eqref{eq:reshuffling} can be obtained from the immediate exchange model~\eqref{eq:exchange} by assuming
 that the two uniform random variables used at each update are not independent but instead satisfy~$\ep_i + \ep_j = 1$.
 Interestingly, this slight change in the interaction rules creates a new behavior.
 Indeed, the numerical simulations in~\cite{heinsalu_patriarca_2014} suggest that the limiting distribution of money now approaches a gamma distribution
 with mean~$T$ and shape parameter two when the population size and the money temperature are large:
 $$ P (m) = \frac{4m}{T^2} \ e^{- 2m/T} \quad \hbox{where} \quad T = M/N. $$
 As previously, this limit has to be understood as follows:
 The probability that a given individual has at least~$m$ dollars at equilibrium approaches
 $$ \int_m^{\infty} \frac{4x}{T^2} \ e^{- 2x/T} \,dx = \bigg(1 + \frac{2m}{T} \bigg) \,e^{-2m / T}. $$
 Shortly after the publication of~\cite{heinsalu_patriarca_2014}, the convergence to the gamma distribution has been proved analytical
 in~\cite{katriel_2015} for an infinite-population version of the immediate exchange model. \\
\indent The third and last model we consider in this paper is inspired from another generalization of the uniform reshuffling model that includes saving
 propensity~\cite{chakraborti_chakrabarti_2000}.
 The two interacting agents now save a fixed fraction~$\lambda$ of their fortune and only the combined remaining fortune is reshuffled randomly between
 the two agents, which makes the uniform reshuffling model the particular case~$\lambda = 0$.
 In equations, an interaction between agents~$i$ and~$j$ results in the update
\begin{equation}
\label{eq:saving}
 \begin{array}{rcl}
   m_i & \n \to \n & m'_i = \lambda \,m_i + \ep \,(1 - \lambda)(m_i + m_j) \vspace*{2pt} \\
   m_j & \n \to \n & m'_j = \lambda \,m_j + (1 - \ep)(1 - \lambda)(m_i + m_j) \end{array} \quad \hbox{where} \quad \ep = \uniform (0, 1).
\end{equation}
 As for the immediate exchange model, the computer simulations performed in~\cite{patriarca_chakraborti_kashi_2004} suggest that the distribution of money
 at equilibrium approaches a gamma distribution with mean~$T$ but the shape parameter~$r$ now depends on the saving propensity.
 More precisely, the probability that a given agent has~$m$ dollars at equilibrium now approaches
 $$ P (m) = \frac{1}{\Gamma (r)} \bigg(\frac{r}{T} \bigg)^r m^{r - 1} \,e^{- rm/T} \quad \hbox{where} \quad
        T = M/N \quad \hbox{and} \quad r = \frac{1 + 2 \lambda}{1 - \lambda} $$
 which again has to be understood as follows:
 The probability that a given individual has at least~$m$ dollars at equilibrium approaches
 $$ \int_m^{\infty} \frac{1}{\Gamma (r)} \bigg(\frac{r}{T} \bigg)^r x^{r - 1} \,e^{- rx/T} \,dx \quad \hbox{where} \quad r = \frac{1 + 2 \lambda}{1 - \lambda}. $$
 Note that the distribution above reduces to the exponential distribution when~$\lambda = 0$, in accordance with the numerical results for the uniform
 reshuffling model in~\cite{dragulescu_yakovenko_2000}.
 Note also that the gamma distribution with shape parameter~$r = 2$, which approximates the limiting distribution of the immediate exchange model, is obtained by
 setting~$\lambda = 1/4$.


\section{Model description}
\label{sec:models}

\indent The models we study in this paper are discrete-state versions of the models~\eqref{eq:reshuffling}--\eqref{eq:saving} that also include a spatial
 structure in the form of local interactions.
\begin{itemize}
 \item Discrete-state versions means that we assume that there is a total of~$M$ coins in the system, where~$M$ is a nonnegative integer, and that individuals
       are characterized by the number of coins they possess, which we again assume to be a nonnegative integer.
       In particular, the fortune of each individual is a discrete quantity rather than a continuous one, and each exchange of money can only result in a
       finite number of outcomes. \vspace*{4pt}
 \item Local interactions, as opposed to global interactions where any two individuals in the system may interact at each time step, means that individuals
       are located on the set of vertices~$\V$ of a graph~$\G = (\V, \E)$ that we assume to be connected, and that only neighbors, i.e., individuals connected
       by an edge~$e \in \E$, can interact to exchange coins.
       The graph~$\G$ has to be thought of as representing a social network where only individuals connected by an edge (friends, business partners, etc.)
       can interact to exchange money.
\end{itemize}
 As previously, we let~$N = \card (\V)$ be the population size.
 Each of the three models is again a discrete-time Markov chain but the state at time~$t \in \N$ is now a spatial configuration
 $$ \xi : \V \to \N \quad \hbox{where} \quad \xi (x) = \hbox{number of coins at vertex~$x$}. $$
 In addition to the fact that the amount of money individuals possess is discrete rather than continuous, the main difference with the non-spatial models
 described in the previous section is that, at each time step, the interacting pair is not selected by choosing a pair uniformly at random but by choosing
 an edge~$e \in \E$ uniformly at random.
 Note that the non-spatial models in the previous section can be viewed as particular cases where~$\G$ is the complete graph with~$N$ vertices. \vspace*{5pt} \\
\noindent{\bf Uniform reshuffling model.}
 The version of the uniform reshuffling model~$(\unif_t)$ we consider evolves in discrete time as follows.
 At each time step, say~$t$, an edge~$(x, y)$ is chosen uniformly at random from the edge set~$\E$, which results in an interaction between the economical
 agents at vertex~$x$ and at vertex~$y$.
 Following~\cite{dragulescu_yakovenko_2000}, we assume that the total amount of coins both agents have at time~$t$ is uniformly redistributed
 between the two agents at time~$t + 1$.
 Since each coin is treated as an indivisible unit, the number of outcomes is finite.
 More precisely, we let
\begin{equation}
\label{eq:reshuffling-1}
  U = \uniform \{0, 1, \ldots, \unif_t (x) + \unif_t (y) \}
\end{equation}
 and update the configuration by setting
\begin{equation}
\label{eq:reshuffling-2}
  \unif_{t + 1} (x) = U \quad \hbox{and} \quad \unif_{t + 1} (y) = \unif_t (x) + \unif_t (y) - U
\end{equation}
 while~$\unif_{t + 1} \equiv \unif_t$ on the set~$\V - \{x, y \}$. Note that
 $$ \unif_t (x) + \unif_t (y) - U = \uniform \{0, 1, \ldots, \unif_t (x) + \unif_t (y) \} \quad \hbox{in distribution}, $$
 indicating that, though~\eqref{eq:reshuffling-2} is not symmetric in~$x$ and~$y$, the numbers of coins the agents at~$x$ and~$y$ receive from the
 interaction are indeed equal in distribution. \vspace*{5pt} \\
\noindent{\bf Immediate exchange model.}
 In our version of the immediate exchange model~$(\imex_t)$, we again choose an edge~$(x, y)$ uniformly at random at each time step, which results in
 an interaction between the two agents incident to the edge.
 Following~\cite{heinsalu_patriarca_2014}, we now assume that the two agents give a (uniform) random number of their coins to the other agent.
 More precisely, we let
\begin{equation}
\label{eq:exchange-1}
  U_1 = \uniform \{0, 1, \ldots, \imex_t (x) \} \quad \hbox{and} \quad U_2 = \uniform \{0, 1, \ldots, \imex_t (y) \}
\end{equation}
 be independent, and update the configuration by setting
\begin{equation}
\label{eq:exchange-2}
  \imex_{t + 1} (x) = \imex_t (x) - U_1 + U_2 \quad \hbox{and} \quad \imex_{t + 1} (y) = \imex_t (x) - U_2 + U_1
\end{equation}
 while~$\imex_{t + 1} \equiv \imex_t$ on the set~$\V - \{x, y \}$. \vspace*{5pt} \\
\noindent{\bf Uniform saving model.}
 As previously, an edge~$(x, y)$ is chosen uniformly at random at each time step, which results in an interaction between the two agents incident to the edge.
 In the original model with saving introduced in~\cite{chakraborti_chakrabarti_2000}, each agent saves a fixed (deterministic) fraction of their fortune
 and the combined remaining amount of money is uniformly redistributed between the two agents.
 In contrast, we add more randomness to the process by assuming that the number of coins each agent saves is also random.
 This results in a model~$(\save_t)$ that combines the previous two types of interactions: random exchange and uniform reshuffling.
 More precisely, we let
\begin{equation}
\label{eq:saving-1}
  U_1 = \uniform \{0, 1, \ldots, \save_t (x) \} \quad \hbox{and} \quad U_2 = \uniform \{0, 1, \ldots, \save_t (y) \}
\end{equation}
 be independent.
 These are the random numbers of coins vertex~$x$ and vertex~$y$ save before the exchange.
 Then, given that~$U_1 = c_x$ and~$U_2 = c_y$, we let
\begin{equation}
\label{eq:saving-2}
  U = \uniform \{0, 1, \ldots, \save_t (x) + \save_t (y) - c_x - c_y \}
\end{equation}
 be the random number of coins vertex~$x$ gets after uniform reshuffling.
 In particular, the new configuration is obtained by setting
\begin{equation}
\label{eq:saving-3}
  \save_{t + 1} (x) = c_x + U \quad \hbox{and} \quad \save_{t + 1} (y) = \save_t (x) + \save_t (y) - c_x - U
\end{equation}
 while~$\save_{t + 1} \equiv \save_t$ on the set~$\V - \{x, y \}$.
 Note that the number of coins at vertex~$y$ after the interaction can be written and interpreted in the following manner:
 $$ \save_{t + 1} (y) = \save_t (x) + \save_t (y) - c_x - U = c_y + (\save_t (x) + \save_t (y) - c_x - c_y - U) $$
 which is the number of coins vertex~$y$ saves before the interaction plus the number of coins vertex~$y$ gets after uniform reshuffling of the coins
 involved in the exchange.


\section{Main results}
\label{sec:results}

\indent Numerical simulations of the uniform reshuffling model~\eqref{eq:reshuffling-1}--\eqref{eq:reshuffling-2} on the complete graph suggest that
 the limiting distribution approaches the exponential distribution
 $$ \frac{1}{T} \ e^{- c/T} \quad \hbox{for all} \quad c = 0, 1, \ldots, M, $$
 shown in Figure~\ref{fig:unif} when the number of vertices and the money temperature are large.
 This is in agreement with the numerical results found for the continuous counterpart~\eqref{eq:reshuffling}.
 Similarly, numerical simulations of the immediate exchange model~\eqref{eq:exchange-1}--\eqref{eq:exchange-2} on the complete graph are in agreement
 with the numerical results found for the continuous counterpart~\eqref{eq:exchange}, suggesting again that the limiting distribution approaches in this
 case the gamma distribution
 $$ \frac{4c}{T^2} \ e^{- 2c/T} \quad \hbox{for all} \quad c = 0, 1, \ldots, M, $$
 shown in Figure~\ref{fig:imex} when the number of vertices and the money temperature are large.
 These results are expected since our versions of the uniform reshuffling and immediate exchange models are good approximations
 of models~\eqref{eq:reshuffling} and~\eqref{eq:exchange} when the number of coins is large.
 Now, in contrast with model~\eqref{eq:saving}, our version of the uniform saving model~\eqref{eq:saving-1}--\eqref{eq:saving-3} does not include any
 parameter measuring the saving propensity.
 As for the immediate exchange model, simulations of the uniform saving model suggest convergence to the gamma distribution with shape parameter two,
 which corresponds to the limit of model~\eqref{eq:saving} with saving propensity~$\lambda = 1/4$. \\
\begin{figure}[h!]
\centering
\includegraphics[width=0.90\textwidth]{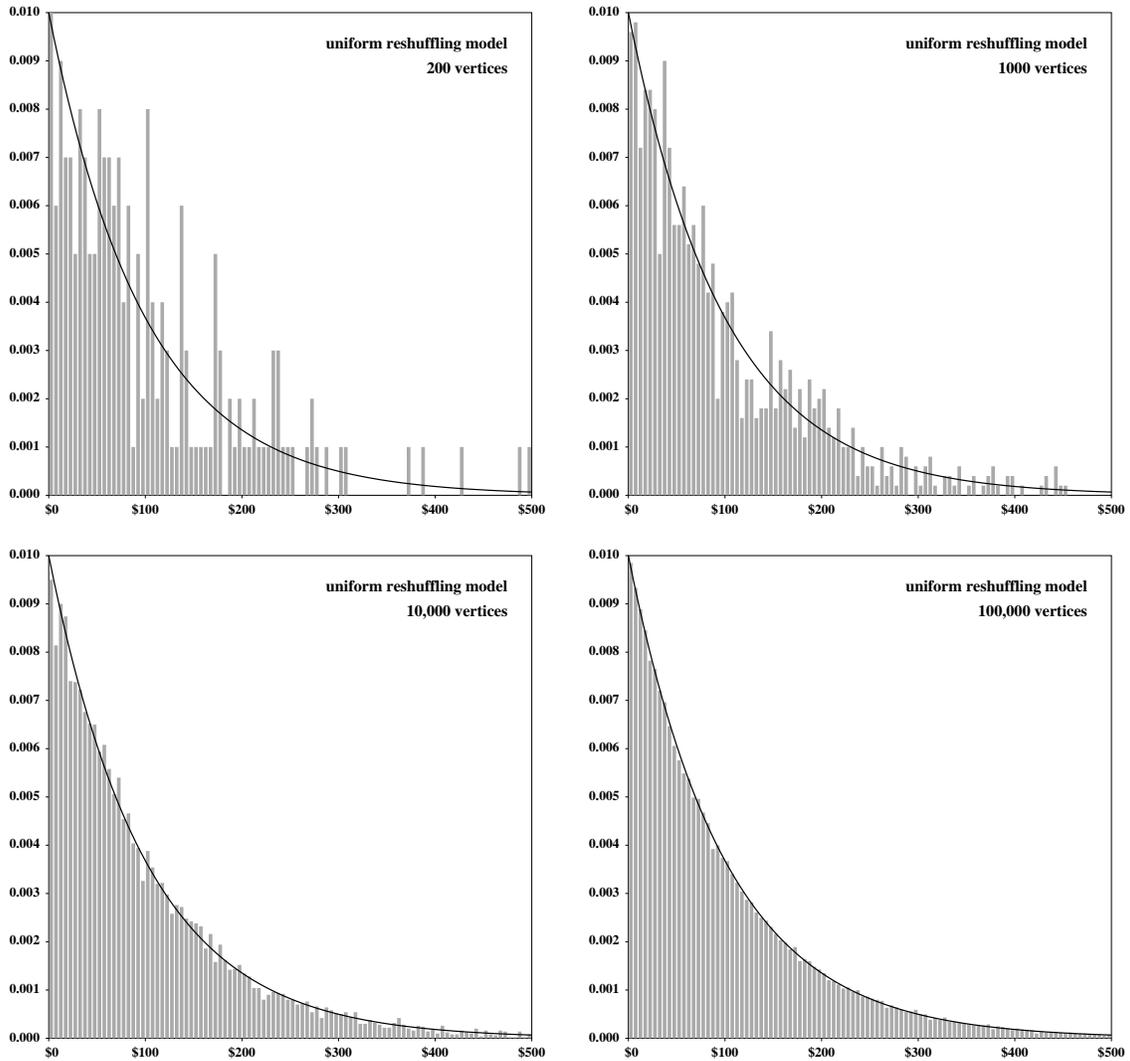}
 \caption{\upshape Simulation results for a single realization of the uniform reshuffling model~\eqref{eq:reshuffling-1}--\eqref{eq:reshuffling-2} on the complete graph.
  The number of vertices used in each simulation is indicated at the top right of the pictures.
  For each of the four simulations, all the vertices start with~$\$100$.
  The gray histograms represent the distribution of money after~$10^6$ updates while the black solid curve is the exponential distribution with mean~$T = 100$.}
\label{fig:unif}
\end{figure}
\begin{figure}[h!]
\centering
\includegraphics[width=0.90\textwidth]{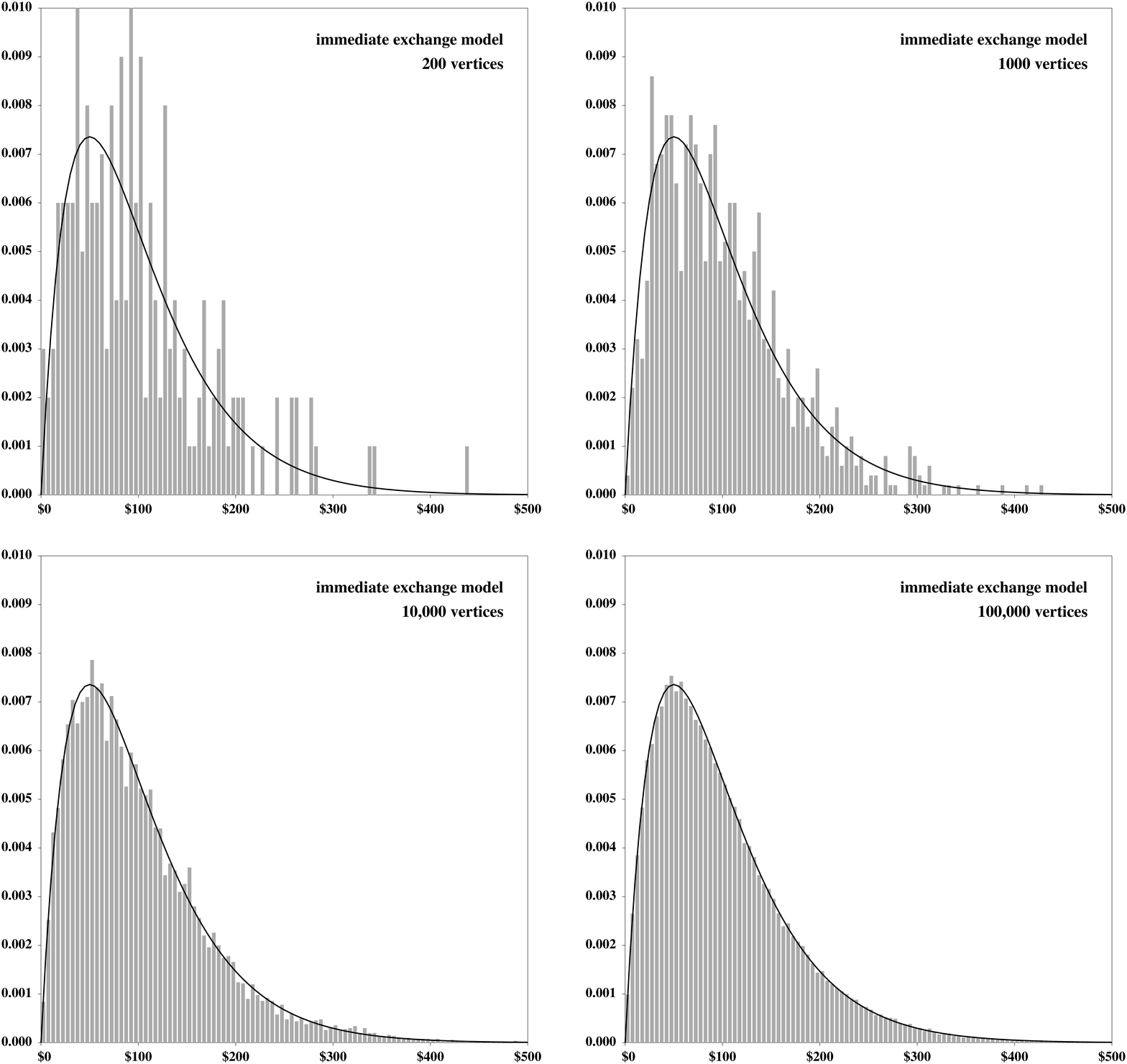}
 \caption{\upshape Simulation results for a single realization of the immediate exchange model~\eqref{eq:exchange-1}--\eqref{eq:exchange-2} on the complete graph.
  The number of vertices used in each simulation is indicated at the top right of the pictures.
  For each of the four simulations, all the vertices start with~$\$100$.
  The gray histograms represent the distribution of money after~$10^6$ updates while the black solid curve is the gamma distribution with mean~$T = 100$ and shape
  parameter two.}
\label{fig:imex}
\end{figure}
\indent Our analytical results for the three models~\eqref{eq:reshuffling-1}--\eqref{eq:saving-3} not only give rigorous proofs of the three conjectures
 above when the number of vertices and the money temperature are large, they also extend these conjectures in several directions:
\begin{enumerate}
 \item The convergence to a distribution of money that approaches the exponential distribution or the gamma distribution holds regardless of the initial
       configuration of the system while the numerical results in~\cite{dragulescu_yakovenko_2000, heinsalu_patriarca_2014, patriarca_chakraborti_kashi_2004}
       assume that each agent starts with~$T$ dollars. \vspace*{4pt}
 \item The convergence to a distribution of money that approaches the exponential distribution or the gamma distribution holds for the general spatial models
       on any connected graphs while the numerical results in~\cite{dragulescu_yakovenko_2000, heinsalu_patriarca_2014, patriarca_chakraborti_kashi_2004}
       focus on the complete graph only. \vspace*{4pt}
 \item The results in 1 and 2 appear as particular cases of more general results that give the exact expression of the distribution of money at equilibrium
       for all possible values of the population size and the money temperature while the conjectures
       in~\cite{dragulescu_yakovenko_2000, heinsalu_patriarca_2014, patriarca_chakraborti_kashi_2004} are only true under the assumption that these two
       quantities are large.
\end{enumerate}
 The level of generality of our results is a good illustration of the advantage of using mathematical tools as opposed to computer simulations that
 cannot be performed for all possible connected graphs with all possible number of vertices containing all possible number of coins starting from all
 possible initial configurations. \\
\indent We now state our results and briefly sketch their proofs.
 For all three models, there is a positive probability that an interaction between~$x$ and~$y$ results in the same number of coins moving
 from~$x \to y$ and from~$y \to x$ and therefore no change after the update.
 This shows that the processes are aperiodic.
 It can also be proved that the three processes are irreducible, which is an intrinsic consequence of the connectedness of the network of interactions.
 These two ingredients together with finiteness of the state space imply that, for each of the three models, there is a unique stationary distribution
 to which the process converges regardless of the initial configuration. \\
\indent For the model with uniform reshuffling, the symmetry of the evolution rules can be used to prove that the probability of a transition
 from~$\xi \to \eta$ is equal to the probability of a transition from~$\eta \to \xi$ for any two configurations~$\xi$ and~$\eta$.
 This implies that the process is doubly stochastic from which it follows that all the configurations are equally likely at equilibrium.
 This can be used to obtain an explicit expression of the distribution of money at equilibrium.
 Some basic algebra also implies that this distribution approaches the exponential distribution with mean~$T$ when~$N$ and~$T$ are both large, in
 agreement with the conjecture in~\cite{dragulescu_yakovenko_2000}.
\begin{theorem}[uniform reshuffling] --
\label{th:reshuffling}
 For all connected graph~$\G$ with~$N$ vertices, regardless of the number~$M$ of coins and the initial configuration,
 $$ \lim_{t \to \infty} P (\unif_t (x) = c) = {M - c + N - 2 \choose N - 2} \bigg/ {M + N - 1 \choose N - 1}. $$
 In particular, when~$N$ and~$T = M/N$ are large,
 $$ \lim_{t \to \infty} P (\unif_t (x) = c) \approx \frac{1}{T} \ e^{- c/T}. $$
\end{theorem}
\indent In contrast, the immediate exchange model is not doubly stochastic.
 However, one can use reversibility to have an implicit expression of the unique stationary distribution.
 Some combinatorial techniques lead to an explicit expression while some basic algebra implies that the distribution of money at equilibrium approaches
 the gamma distribution with mean~$T$ and shape parameter two when~$N$ and~$T$ are both large, in agreement with the conjecture
 in~\cite{heinsalu_patriarca_2014}.
\begin{theorem}[immediate exchange] --
\label{th:exchange}
 For all connected graph~$\G$ with~$N$ vertices, regardless of the number~$M$ of coins and the initial configuration,
 $$ \lim_{t \to \infty} P (\imex_t (x) = c) = (c + 1) {M - c + 2N - 3 \choose 2N - 3} \bigg/ {M + 2N - 1 \choose 2N - 1}. $$
 In particular, when~$N$ and~$T = M/N$ are large,
 $$ \lim_{t \to \infty} P (\imex_t (x) = c) \approx \frac{4c}{T^2} \ e^{- 2c/T}. $$
\end{theorem}
\indent Turning to the uniform saving model, though its evolution rules are somewhat different from the evolution rules of the immediate exchange model,
 it can be proved that their respective stationary distributions satisfy the same detailed balance equation and therefore are equal.
 In particular, our previous theorem extends to the uniform saving model.
\begin{theorem}[uniform saving] --
\label{th:saving}
 For all connected graph~$\G$ with~$N$ vertices, regardless of the number~$M$ of coins and the initial configuration,
 $$ \lim_{t \to \infty} P (\save_t (x) = c) = (c + 1) {M - c + 2N - 3 \choose 2N - 3} \bigg/ {M + 2N - 1 \choose 2N - 1}. $$
 In particular, when~$N$ and~$T = M/N$ are large,
 $$ \lim_{t \to \infty} P (\save_t (x) = c) \approx \frac{4c}{T^2} \ e^{- 2c/T}. $$
\end{theorem}
\indent The rest of this paper is devoted to proofs.
 Theorem~\ref{th:reshuffling} is proved in Section~\ref{sec:reshuffling}.
 Sections~\ref{sec:exchange} and~\ref{sec:saving} focus on the reversibility of the immediate exchange model and the uniform saving model, respectively,
 and give the corresponding detailed balance equations.
 Section~\ref{sec:exchange-saving} gives the common final step to complete the proof of Theorems~\ref{th:exchange} and~\ref{th:saving}.


\section{Proof of Theorem~\ref{th:reshuffling}}
\label{sec:reshuffling}

\indent This section is devoted to the proof of Theorem~\ref{th:reshuffling} about the limiting distribution of money for the uniform
 reshuffling model~\eqref{eq:reshuffling-1}--\eqref{eq:reshuffling-2}.
 The proof relies on the following two key ingredients:
\begin{enumerate}
 \item There exists a unique stationary distribution~$\pi_{\unif}$ to which the uniform reshuffling model converges starting from any initial configuration. \vspace*{4pt}
 \item The uniform distribution on the set of all possible configurations is stationary.
\end{enumerate}
 With these two preliminary results in hand, the theorem follows using some basic combinatorics and some basic algebra.
 From now on, we let
 $$ \begin{array}{l} \C_{N, M} = \{\xi : \V \to \N \ \hbox{such that} \ \sum_{x \in \V} \xi (x) = M \} \end{array} $$
 be the set of all possible configurations with exactly~$M$ coins.
 Also, for every vertex~$x \in \V$ and every configuration~$\xi : \V \to \N$, we let
 $$ \xi^x (z) = \xi (z) + \ind \{z = x \} \quad \hbox{for all} \quad z \in \V $$
 be the configuration obtained from~$\xi$ by adding one coin at vertex~$x$.
 We now prove existence and uniqueness of the stationary distribution.
\begin{lemma} --
\label{lem:reshuffling-limit}
 There is a unique stationary distribution~$\pi_{\unif}$ and
 $$ \lim_{t \to \infty} P_{\eta} (\unif_t = \xi) = \pi_{\unif} (\xi) \quad \hbox{for all} \quad \xi, \eta \in \C_{N, M}. $$
\end{lemma}
\begin{proof}
 According to~\cite[Theorem~7.7]{lanchier_2017a}, it suffices to prove that the process is finite, irreducible and aperiodic.
 Finiteness is obvious while aperiodicity follows from the fact that
 $$ \begin{array}{rcl}
     P (\unif_{t + 1} = \xi \,| \,\unif_t = \xi) & \n = \n &
        \displaystyle \frac{1}{\card (\E)} \sum_{(x, y) \in \E} P (\uniform \{0, 1, \ldots, \xi (x) + \xi (y) \} = \xi (x)) \vspace*{4pt} \\ & \n = \n &
        \displaystyle \frac{1}{\card (\E)} \sum_{(x, y) \in \E} \bigg(\frac{1}{\xi (x) + \xi (y) + 1} \bigg) \geq
        \displaystyle \bigg(\frac{1}{M + 1} \bigg) > 0 \end{array} $$
 for every configuration~$\xi \in \C_{N, M}$.
 To prove that the process is also irreducible, let~$x, y \in \V$.
 Since the graph is connected, there exists a path
 $$ (x_0, x_1, \ldots, x_t) \subset \V \quad \hbox{such that} \quad x_0 = x, \ x_t = y \ \hbox{and} \ t < N. $$
 In particular, for all~$\xi \in \C_{N, M - 1}$,
\begin{equation}
\label{eq:reshuffling-limit-1}
  \begin{array}{l}
    \displaystyle P (\unif_t = \xi^y \,| \,\unif_0 = \xi^x) \geq
    \displaystyle \prod_{s = 0}^{t - 1} \ P (\unif_{s + 1} = \xi^{x_{s + 1}} \,| \,\unif_s = \xi^{x_s}) \vspace*{4pt} \\ \hspace*{40pt} =
    \displaystyle \prod_{s = 0}^{t - 1} \ \bigg(\frac{1}{\card (\E)} \bigg) \bigg(\frac{1}{\xi (x_s) + \xi (x_{s + 1}) + 2} \bigg) \geq
    \displaystyle \bigg(\frac{1}{N^2 \,(M + 1)} \bigg)^N > 0. \end{array}
\end{equation}
 Let~$m \leq M$.
 We deduce from~\eqref{eq:reshuffling-limit-1} by induction that, for all
 $$ \xi \in \C_{N, M - m} \quad \hbox{and} \quad x, y \in \V^m $$
 there exists~$t < mN$ such that
\begin{equation}
\label{eq:reshuffling-limit-2}
  P (\unif_t = (\cdots (\xi^{y_1})^{y_2} \cdots )^{y_m} \,| \,\unif_0 = (\cdots (\xi^{x_1})^{x_2} \cdots )^{x_m}) \geq
     \bigg(\frac{1}{N^2 \,(M + 1)} \bigg)^{mN}.
\end{equation}
 Since any~$\xi, \eta \in \C_{N, M}$ can be obtained from the configuration with zero coin by adding~$M$ coins at the appropriate vertices, it follows
 from~\eqref{eq:reshuffling-limit-2} with~$m = M$ that
 $$ P (\unif_t = \eta \,| \,\unif_0 = \xi) \geq \bigg(\frac{1}{N^2 \,(M + 1)} \bigg)^{MN} > 0 \quad \hbox{for some} \quad t < MN. $$
 This shows that the process is irreducible and completes the proof.
\end{proof} \\ \\
 The next lemma shows that the uniform reshuffling model is doubly stochastic, from which we deduce that the unique stationary distribution~$\pi_{\unif}$ from the previous
 lemma is the uniform distribution on the set of configurations~$\C_{N, M}$.
\begin{lemma} --
\label{lem:reshuffling-uniform}
 The unique stationary distribution is~$\pi_{\unif} = \uniform (\C_{N, M})$.
\end{lemma}
\begin{proof}
 Let~$\xi, \eta \in \C_{N, M}$.
 Note that~$P (\unif_{t + 1} = \eta \,| \,\unif_t = \xi) > 0$ if and only if
 $$ \xi \equiv \eta \ \hbox{on} \ \V - \{x, y \} \quad \hbox{and} \quad \xi (x) + \xi (y) = \eta (x) + \eta (y) $$
 for some~$(x, y) \in \E$, in which case we have
 $$ \begin{array}{rcl}
      P (\unif_{t + 1} = \eta \,| \,\unif_t = \xi) & \n = \n &
         \displaystyle \frac{1}{\card (\E)} \ P (\uniform \{0, 1, \ldots, \xi (x) + \xi (y) \} = \eta (x)) \vspace*{4pt} \\ & \n = \n &
         \displaystyle \frac{1}{\card (\E)} \ \frac{1}{\xi (x) + \xi (y) + 1}. \end{array} $$
 In particular, either~$P (\unif_{t + 1} = \eta \,| \,\unif_t = \xi) = P (\unif_{t + 1} = \xi \,| \,\unif_t = \eta) = 0$ or
 $$ \begin{array}{rcl}
      P (\unif_{t + 1} = \eta \,| \,\unif_t = \xi) & \n = \n &
         \displaystyle \frac{1}{\card (\E)} \ \frac{1}{\xi (x) + \xi (y) + 1} \vspace*{4pt} \\ & \n = \n &
         \displaystyle \frac{1}{\card (\E)} \ \frac{1}{\eta (x) + \eta (y) + 1} = P (\unif_{t + 1} = \xi \,| \,\unif_t = \eta). \end{array} $$
 This shows that the transition matrix of the process is symmetric and so doubly stochastic.
 Therefore, it follows from \cite[Section~7.3]{lanchier_2017a} that the uniform distribution on the set of configurations is stationary.
 By the uniqueness of the stationary distribution~$\pi_{\unif}$ established in the previous lemma, we conclude that~$\pi_{\unif} = \uniform (\C_{N, M})$.
\end{proof} \\ \\
 With Lemmas~\ref{lem:reshuffling-limit} and~\ref{lem:reshuffling-uniform} in hand, we are now ready to prove the theorem. \\ \\
\begin{proofof}{Theorem~\ref{th:reshuffling}}
 This is similar to the proofs of Lemma~4 and Theorem~1 in~\cite{lanchier_2017b} that we briefly recall.
 First, we note that
 $$ \card (\C_{N, M}) = \card \,\{c \in \N^N : c_1 + \cdots + c_N = M \} = {M + N - 1 \choose N - 1}. $$
 Since in addition all the configurations are equally likely under~$\pi_{\unif}$ according to Lemma~\ref{lem:reshuffling-uniform}, and since there
 are~$\card (\C_{N - 1, M - c})$ configurations with exactly~$c$ coins at vertex~$x$,
 $$ \lim_{t \to \infty} P (\unif_t (x) = c) = \frac{\card (\C_{N - 1, M - c})}{\card (\C_{N, M})} = {M - c + N - 2 \choose N - 2} \bigg/ {M + N - 1 \choose N - 1}. $$
 This shows the first part of the theorem.
 In addition, when~$N$ and~$T$ are large,
 $$ \begin{array}{rcl}
    \displaystyle \lim_{t \to \infty} P (\unif_t (x) = c) & \n = \n &
    \displaystyle \frac{(M - c + N - 2) \cdots (M - c + 1)}{(M + N - 1) \cdots (M + 1)} \ \frac{(N - 1)!}{(N - 2)!} \vspace*{8pt} \\ & \n = \n &
    \displaystyle \frac{(M - c + N - 2) \cdots (M - c + 1)}{(M + N - 2) \cdots (M + 1)} \ \frac{(N - 1)}{(M + N - 1)} \vspace*{8pt} \\ & \n \approx \n &
    \displaystyle \bigg(\frac{N}{NT} \bigg) \bigg(1 - \frac{c}{NT} \bigg)^N \approx \frac{1}{T} \ e^{- c/T}. \end{array} $$
 This shows the second part of the theorem.
\end{proofof}


\section{Reversibility of the immediate exchange model}
\label{sec:exchange}

\indent This section collects preliminary results about the immediate exchange model~\eqref{eq:exchange-1}--\eqref{eq:exchange-2} that will be useful to
 prove Theorem~\ref{th:exchange}.
 As for the uniform reshuffling model, the first step is to show that there exists a unique stationary distribution~$\pi_{\imex}$ to which the immediate exchange model
 converges starting from any initial configuration.
 Contrary to the uniform reshuffling model, the process is not doubly stochastic and so the uniform distribution is no longer stationary.
 However, an implicit expression of the (unique) stationary distribution can be found using reversibility.
\begin{lemma} --
\label{lem:exchange-limit}
 There is a unique stationary distribution~$\pi_{\imex}$ and
 $$ \lim_{t \to \infty} P_{\eta} (\imex_t = \xi) = \pi_{\imex} (\xi) \quad \hbox{for all} \quad \xi, \eta \in \C_{N, M}. $$
\end{lemma}
\begin{proof}
 As for the uniform reshuffling model, it suffices to establish finiteness, irreducibility and aperiodicity.
 Finiteness is again obvious.
 Letting
 $$ U (z, \xi) = \uniform \{0, 1, \ldots, \xi (z) \} \quad \hbox{for all} \quad z \in \V $$
 be independent, aperiodicity follows from the fact that
 $$ \begin{array}{rcl}
     P (\imex_{t + 1} = \xi \,| \,\imex_t = \xi) & \n = \n &
        \displaystyle \frac{1}{\card (\E)} \sum_{(x, y) \in \E} P (U (x, \xi) = U (y, \xi)) \vspace*{4pt} \\ & \n = \n &
        \displaystyle \frac{1}{\card (\E)} \sum_{(x, y) \in \E} \frac{\min (\xi (x) + 1, \xi (y) + 1)}{(\xi (x) + 1)(\xi (y) + 1)} \geq
        \displaystyle \bigg(\frac{1}{M + 1} \bigg) > 0 \end{array} $$
 for every~$\xi \in \C_{N, M}$.
 Also, letting~$(x, y) \in \E$ and~$\xi \in \C_{N, M - 1}$,
 $$ \begin{array}{rcl}
      P (\imex_{t + 1} = \xi^y \,| \,\imex_t = \xi^x) & \n = \n &
         \displaystyle \frac{P (\uniform \{0, 1, \ldots, \xi (x) + 1 \} = U (y, \xi) + 1)}{\card (\E)} \vspace*{4pt} \\ & \n = \n &
         \displaystyle \frac{1}{\card (\E)} \ \frac{\min (\xi (x) + 1, \xi (y) + 1)}{(\xi (x) + 2)(\xi (y) + 1)} \geq
         \displaystyle \frac{1}{N^2 \,(2M + 2)} > 0. \end{array} $$
 Repeating the proof of Lemma~\ref{lem:reshuffling-limit}, we deduce that, for all~$\xi, \eta \in \C_{N, M}$,
 $$ P (\imex_t = \eta \,| \,\imex_0 = \xi) = \bigg(\frac{1}{N^2 \,(2M + 2)} \bigg)^{MN} > 0 \quad \hbox{for some} \quad t < MN, $$
 which shows irreducibility.
\end{proof} \\ \\
 We now give an implicit expression of~$\pi_{\imex}$ using reversibility.
\begin{lemma} --
\label{lem:exchange-reversible}
 The distribution~$\pi_{\imex}$ is reversible and
\begin{equation}
\label{eq:exchange-reversible-0}
  \pi_{\imex} (\xi) = \frac{\mu (\xi)}{\displaystyle \sum_{\eta \in \C_{N, M}} \mu (\eta)} \quad \hbox{where} \quad \mu (\xi) = \prod_{z \in \V} \,(\xi (z) + 1).
\end{equation}
\end{lemma}
\begin{proof}
 Let~$\xi \neq \eta$ in~$\C_{N, M}$ and assume that, for some~$(x, y) \in \E$,
\begin{equation}
\label{eq:exchange-reversible-1}
  \xi \equiv \eta \ \hbox{on} \ \V - \{x, y \} \quad \hbox{and} \quad \xi (x) + \xi (y) = \eta (x) + \eta (y).
\end{equation}
 Letting~$U (z, \xi) = \uniform \{0, 1, \ldots, \xi (z) \}$ be independent, we have
 $$ \begin{array}{rcl}
      P (\imex_{t + 1} = \eta \,| \,\imex_t = \xi) & \n = \n &
         \displaystyle \frac{P (\xi (x) + U (y, \xi) - U (x, \xi) = \eta (x))}{\card (\E)} \vspace*{4pt} \\ & \n = \n &
         \displaystyle \frac{1}{\card (\E)} \ \sum_{c_x = 0}^{\xi (x)} \ \sum_{c_y = 0}^{\xi (y)} \ \frac{\ind \{c_x = \xi (x) - \eta (x) + c_y \}}{(\xi (x) + 1)(\xi (y) + 1)} \vspace*{4pt} \\ & \n = \n &
         \displaystyle \frac{1}{\card (\E)} \ \sum_{c_x = 0}^{\xi (x)} \ \frac{\ind \{\xi (x) - \eta (x) \leq c_x \leq \xi (x) - \eta (x) + \xi (y) \}}{(\xi (x) + 1)(\xi (y) + 1)}. \end{array} $$
 Since~$\xi (x) - \eta (x) + \xi (y) = \eta (y)$, we get
 $$ \begin{array}{rcl}
      Q_{x,y} (\xi, \eta) & \n = \n & \card (\E) \,(\xi (x) + 1)(\xi (y) + 1) \,P (\imex_{t + 1} = \eta \,| \,\imex_t = \xi) \vspace*{4pt} \\ & \n = \n &
    \min (\xi (x), \xi (x) - \eta (x) + \xi (y)) - \max (0, \xi (x) - \eta (x)) + 1 \vspace*{4pt} \\ & \n = \n &
    \min (\xi (x), \eta (y)) + \min (\xi (x), \eta (x)) - \xi (x) + 1. \end{array} $$
 Using also that~$\eta (y) - \xi (x) = \xi (y) - \eta (x)$,
 $$ \begin{array}{rcl}
      Q_{x,y} (\xi, \eta) & \n = \n & \min (\xi (x), \eta (y)) + \min (\xi (x), \eta (x)) - \xi (x) + 1 \vspace*{4pt} \\
                          & \n = \n & \min (0, \eta (y) - \xi (x)) + \min (\xi (x), \eta (x)) + 1 \vspace*{4pt} \\
                          & \n = \n & \min (0, \xi (y) - \eta (x)) + \min (\xi (x), \eta (x)) + 1 \vspace*{4pt} \\
                          & \n = \n & \min (\eta (x), \xi (y)) + \min (\eta (x), \xi (x)) - \eta (x) + 1 = Q_{x,y} (\eta, \xi) \end{array} $$
 from which it follows that
\begin{equation}
\label{eq:exchange-reversible-2}
  \begin{array}{l}
    (\xi (x) + 1)(\xi (y) + 1) \,P (\imex_{t + 1} = \eta \,| \,\imex_t = \xi) \vspace*{4pt} \\ \hspace*{80pt} =
    (\eta (x) + 1)(\eta (y) + 1) \,P (\imex_{t + 1} = \xi \,| \,\imex_t = \eta). \end{array}
\end{equation}
 If on the contrary condition~\eqref{eq:exchange-reversible-1} is not satisfied, since there are only two neighbors exchanging money at each time step, we must have
\begin{equation}
\label{eq:exchange-reversible-3}
  P (\imex_{t + 1} = \eta \,| \,\imex_t = \xi) = P (\imex_{t + 1} = \xi \,| \,\imex_t = \eta) = 0.
\end{equation}
 Combining~\eqref{eq:exchange-reversible-2} and~\eqref{eq:exchange-reversible-3}, we conclude that, in any case,
 $$ \mu (\xi) \,P (\imex_{t + 1} = \eta \,| \,\imex_t = \xi) = \mu (\eta) \,P (\imex_{t + 1} = \xi \,| \,\imex_t = \eta) \quad \hbox{where} \quad
    \mu (\xi) = \prod_{z \in \V} \,(\xi (z) + 1). $$
 By uniqueness, this implies that~$\pi_{\imex}$ is reversible and satisfies~\eqref{eq:exchange-reversible-0}.
\end{proof}


\section{Reversibility of the uniform saving model}
\label{sec:saving}

\indent The objective of this section is to prove that Lemmas~\ref{lem:exchange-limit} and~\ref{lem:exchange-reversible} in the previous section
 also hold for the uniform saving model~\eqref{eq:saving-1}--\eqref{eq:saving-3}.
 The main ideas behind the proofs are the same as for the immediate exchange model but the technical details are somewhat different.
\begin{lemma} --
\label{lem:saving-limit}
 There is a unique stationary distribution~$\pi_{\save}$ and
 $$ \lim_{t \to \infty} P_{\eta} (\save_t = \xi) = \pi_{\save} (\xi) \quad \hbox{for all} \quad \xi, \eta \in \C_{N, M}. $$
\end{lemma}
\begin{proof}
 Let~$\xi \in \C_{N, M}$ and let
 $$ U (z, \xi) = \uniform \{0, 1, \ldots, \xi (z) \} \quad \hbox{and} \quad U_c = \uniform \{0, 1, \ldots, c \} $$
 be independent for all~$z \in \V$ and~$c \in \N$. Then,
 $$ \begin{array}{rcl}
     P (\save_{t + 1} = \xi \,| \,\save_t = \xi) & \n \geq \n &
        \displaystyle \frac{1}{\card (\E)} \sum_{(x, y) \in \E} P (U (x, \xi) = \xi (x), U (y, \xi) = \xi (y)) \vspace*{4pt} \\ & \n = \n &
        \displaystyle \frac{1}{\card (\E)} \sum_{(x, y) \in \E} \frac{1}{(\xi (x) + 1)(\xi (y) + 1)} \geq
        \displaystyle \bigg(\frac{1}{M + 1} \bigg)^2 > 0 \end{array} $$
 so the process is aperiodic.
 Also, letting~$(x, y) \in \E$ and~$\xi \in \C_{N, M - 1}$,
 $$ \begin{array}{rcl}
      P (\save_{t + 1} = \xi^y \,| \,\save_t = \xi^x) & \n \geq \n &
         \displaystyle \frac{P (\uniform \{0, 1, \ldots, \xi (x) + 1 \} = \xi (x), U (y, \xi) = \xi (y), U_1 = 0)}{\card (\E)} \vspace*{4pt} \\ & \n = \n &
         \displaystyle \frac{1}{\card (\E)} \ \frac{1}{2 \,(\xi (x) + 2)(\xi (y) + 1)} \geq
         \bigg(\frac{1}{N (M + 2)} \bigg)^2 > 0. \end{array} $$
 Repeating the proof of Lemma~\ref{lem:reshuffling-limit}, we deduce that, for all~$\xi, \eta \in \C_{N, M}$,
 $$ P (\save_t = \eta \,| \,\save_0 = \xi) = \bigg(\frac{1}{N (M + 2)} \bigg)^{2MN} > 0 \quad \hbox{for some} \quad t < MN, $$
 so the process is irreducible.
 As previously, convergence to a unique stationary distribution follows from the fact that the process is finite, irreducible and aperiodic.
\end{proof}
\begin{lemma} --
\label{lem:saving-reversible}
 The distribution~$\pi_{\save}$ is reversible and
\begin{equation}
\label{eq:saving-reversible-0}
  \pi_{\save} (\xi) = \frac{\mu (\xi)}{\displaystyle \sum_{\eta \in \C_{N, M}} \mu (\eta)} \quad \hbox{where} \quad \mu (\xi) = \prod_{z \in \V} \,(\xi (z) + 1).
\end{equation}
\end{lemma}
\begin{proof}
 Let~$\xi \neq \eta$ in~$\C_{N, M}$ be two configurations.
 As for the uniform reshuffling and immediate exchange models, when condition~\eqref{eq:exchange-reversible-1} is not satisfied,
\begin{equation}
\label{eq:saving-reversible-1}
  P (\save_{t + 1} = \eta \,| \,\save_t = \xi) = P (\save_{t + 1} = \xi \,| \,\save_t = \eta) = 0.
\end{equation}
 To study the transition probability when~\eqref{eq:exchange-reversible-1} holds, let
 $$ U (z, \xi) = \uniform \{0, 1, \ldots, \xi (z) \} \quad \hbox{and} \quad U_c = \uniform \{0, 1, \ldots, c \} $$
 be independent for all~$z \in \V$ and~$c \in \N$.
 By conditioning on all the possible values of~$U (x, \xi)$ and~$U (y, \xi)$ and using independence, we get
\begin{equation}
\label{eq:saving-reversible-2}
  P (\save_{t + 1} = \eta \,| \,\save_t = \xi) =
     \frac{1}{\card (\E)} \ \sum_{c_x = 0}^{\xi (x)} \ \sum_{c_y = 0}^{\xi (y)} \
     \frac{P (c_x + U_{\xi (x) + \xi (y) - c_x - c_y} = \eta (x))}{(\xi (x) + 1)(\xi (y) + 1)}.
\end{equation}
 By conditioning on all the possible values of~$U_{\xi (x) + \xi (y) - c_x - c_y}$ and using again independence, the numerator in the sum above can be written as
\begin{equation}
\label{eq:saving-reversible-3}
  \begin{array}{l}
    P (c_x + U_{\xi (x) + \xi (y) - c_x - c_y} = \eta (x)) =
       \displaystyle \frac{\ind \{c_x \leq \eta (x) \leq \xi (x) + \xi (y) - c_y \}}{\xi (x) + \xi (y) - c_x - c_y + 1} \vspace*{8pt} \\ \hspace*{60pt} =
       \displaystyle \frac{\ind \{c_x \leq \eta (x) \leq \eta (x) + \eta (y) - c_y \}}{\xi (x) + \xi (y) - c_x - c_y + 1} =
       \displaystyle \frac{\ind \{c_x \leq \eta (x) \} \,\ind \{c_y \leq \eta (y) \}}{\xi (x) + \xi (y) - c_x - c_y + 1}. \end{array}
\end{equation}
 Combining~\eqref{eq:saving-reversible-2} and~\eqref{eq:saving-reversible-3}, we obtain that
 $$ Q_{x,y} (\xi, \eta) = \card (\E) \,(\xi (x) + 1)(\xi (y) + 1) \,P (\save_{t + 1} = \eta \,| \,\save_t = \xi) $$
 can be written using symmetry as
 $$ \begin{array}{rcl}
      Q_{x,y} (\xi, \eta) & \n = \n &
    \displaystyle \sum_{c_x = 0}^{\xi (x)} \ \sum_{c_y = 0}^{\xi (y)} \ \frac{\ind \{c_x \leq \eta (x) \} \,\ind \{c_y \leq \eta (y) \}}{\xi (x) + \xi (y) - c_x - c_y + 1} \vspace*{4pt} \\ & \n = \n &
    \displaystyle \sum_{c_x = 0}^{\xi (x) \wedge \eta (x)} \ \sum_{c_y = 0}^{\xi (y) \wedge \eta (y)} \bigg(\frac{1}{\xi (x) + \xi (y) - c_x - c_y + 1} \bigg) \vspace*{4pt} \\ & \n = \n &
    \displaystyle \sum_{c_x = 0}^{\eta (x) \wedge \xi (x)} \ \sum_{c_y = 0}^{\eta (y) \wedge \xi (y)} \bigg(\frac{1}{\eta (x) + \eta (y) - c_x - c_y + 1} \bigg) =
      Q_{x,y} (\eta, \xi) \end{array} $$
 from which it follows that
\begin{equation}
\label{eq:saving-reversible-4}
  \begin{array}{l}
    (\xi (x) + 1)(\xi (y) + 1) \,P (\save_{t + 1} = \eta \,| \,\save_t = \xi) \vspace*{4pt} \\ \hspace*{80pt} =
    (\eta (x) + 1)(\eta (y) + 1) \,P (\save_{t + 1} = \xi \,| \,\save_t = \eta). \end{array}
\end{equation}
 Combining~\eqref{eq:saving-reversible-1} and~\eqref{eq:saving-reversible-4}, we conclude that, in any case,
 $$ \mu (\xi) \,P (\save_{t + 1} = \eta \,| \,\save_t = \xi) = \mu (\eta) \,P (\save_{t + 1} = \xi \,| \,\save_t = \eta) \quad \hbox{where} \quad
    \mu (\xi) = \prod_{z \in \V} \,(\xi (z) + 1), $$
 showing that~$\pi_{\save}$ is reversible and satisfies~\eqref{eq:saving-reversible-0}.
\end{proof}


\section{Proof of Theorems~\ref{th:exchange} and~\ref{th:saving}}
\label{sec:exchange-saving}

\indent Lemmas~\ref{lem:exchange-limit}--\ref{lem:saving-reversible} in the previous two sections imply that, though the evolution rules of the immediate exchange
 model and of the uniform saving model are different, both processes converge to the same stationary distribution~$\pi = \pi_{\imex} = \pi_{\save}$ which is characterized by
 $$ \pi (\xi) = \frac{\mu (\xi)}{\displaystyle \sum_{\eta \in \C_{N, M}} \mu (\eta)} \quad \hbox{where} \quad \mu (\xi) = \prod_{z \in \V} \,(\xi (z) + 1). $$
 To complete the proof of Theorems~\ref{th:exchange} and~\ref{th:saving}, the last step is to find a more explicit expression of the stationary distribution by
 computing the denominator
 $$ \Lambda (N, M) = \sum_{\xi \in \C_{N, M}} \prod_{z \in \V} \,(\xi (z) + 1) = \sum_{c_1 + \cdots + c_N = M} (c_1 + 1)(c_2 + 1) \cdots (c_N + 1). $$
 To compute~$\Lambda (N, M)$, we start with the following technical lemma.
\begin{lemma} --
\label{lem:combinatorics}
 For all~$M, K \in \N$,
 $$ S (M, K) = \sum_{c = 0}^M \ (c + 1) {M - c + K \choose K} = {M + K + 2 \choose K + 2}. $$
\end{lemma}
\begin{proof}
 We prove the result by induction on~$M + K$.
 The fact that
 $$ \begin{array}{rcl}
      S (0, K) & \n = \n & \displaystyle \sum_{c = 0}^0 \ (c + 1) {0 - c + K \choose K} = {K \choose K} = 1 = {0 + K + 2 \choose K + 2} \vspace*{4pt} \\
      S (M, 0) & \n = \n & \displaystyle \sum_{c = 0}^M \ (c + 1) {M - c + 0 \choose 0} = \sum_{c = 0}^M \ (c + 1) = \frac{(M + 1)(M + 2)}{2} = {M + 0 + 2 \choose 0 + 2} \end{array} $$
 shows that the result holds when~$M = 0$ or~$K = 0$.
 Now, let~$m \in \N^*$ and assume that the result holds whenever~$M + K < m$.
 Using the well-known identity
 $$ {n \choose k} = {n - 1 \choose k} + {n - 1 \choose k - 1} \quad \hbox{for all} \quad 1 \leq k < n $$
 consecutively in the following two cases
 $$ \begin{array}{rcl}
          n = M - c + K \ \ \hbox{and} \ \ k = K & \hbox{with} & K \geq 1 \ \ \hbox{and} \ \ M > c \vspace*{4pt} \\
      n = M + K + 2 \ \ \hbox{and} \ \ k = K + 2 & \hbox{with} & K \geq 1 \ \ \hbox{and} \ \ M \geq 1 \end{array} $$
 and assuming that~$M + K = m$ with~$M, K \geq 1$, we get
 $$ \begin{array}{rcl}
     S (M, K) & \n = \n & \displaystyle \sum_{c = 0}^{M - 1} \ (c + 1) \left[{(M - 1) - c + K \choose K} + {M - c + (K - 1) \choose K - 1} \right] + (M + 1) \vspace*{4pt} \\
              & \n = \n & \displaystyle  S (M - 1, K) + \sum_{c = 0}^M \ (c + 1) {M - c + (K - 1) \choose K - 1} = S (M - 1, K) + S (M, K - 1) \vspace*{6pt} \\
              & \n = \n & \displaystyle {M + K + 2 - 1 \choose K + 2} + {M + K + 2 - 1 \choose K + 2 - 1} = {M + K + 2 \choose K + 2}. \end{array} $$
 This completes the proof.
\end{proof} \\ \\
 Using Lemma~\ref{lem:combinatorics}, we can now compute~$\Lambda (N, M)$.
\begin{lemma} --
\label{lem:total-mass}
 For all~$N, M \geq 1$, we have
 $$ \Lambda (N, M) = \sum_{c_1 + \cdots + c_N = M} (c_1 + 1)(c_2 + 1) \cdots (c_N + 1) = {M + 2N - 1 \choose 2N - 1}. $$
\end{lemma}
\begin{proof}
 We prove the result by induction on~$N$.
 Observing that
 $$ \Lambda (1, M) = \sum_{c_1 = M} (c_1 + 1) = (M + 1) = {M + 2 - 1 \choose 2 - 1} $$
 shows that the result holds for~$N = 1$.
 Now, fix~$N \geq 2$ and assume that the result holds for~$N - 1$ vertices.
 Decomposing according to the possible values of~$c_N$, we get
 $$ \begin{array}{rcl}
    \Lambda (N, M) & \n = \n & \displaystyle \sum_{c_N = 0}^M \ (c_N + 1) \sum_{c_1 + \cdots + c_{N - 1} = M - c_N} (c_1 + 1)(c_2 + 1) \cdots (c_{N - 1} + 1) \vspace*{4pt} \\
                  & \n = \n & \displaystyle \sum_{c = 0}^M \ (c + 1) \,\Lambda (N - 1, M - c) = \sum_{c = 0}^M \ (c + 1) {M - c + 2N - 3 \choose 2N - 3}. \end{array} $$
 Finally, applying Lemma~\ref{lem:combinatorics}, we obtain
 $$ \Lambda (N, M) = S (M, 2N - 3) = {M + 2N - 3 + 2 \choose 2N - 3 + 2} = {M + 2N - 1 \choose 2N - 1}, $$
 which completes the proof.
\end{proof} \\ \\
 We are now ready to prove the theorems. \\ \\
\begin{proofof}{Theorems~\ref{th:exchange} and~\ref{th:saving}}
 Combining~Lemmas~\ref{lem:exchange-limit} and~\ref{lem:exchange-reversible}, we obtain that, regardless of the initial configuration and regardless of the choice of
 vertex~$x \in \V$,
 $$ \begin{array}{l}
    \displaystyle \lim_{t \to \infty} P (\imex_t (x) = c) =
    \sum_{\xi : \xi (x) = c} \pi (\xi) \vspace*{4pt} \\ \hspace*{20pt} =
    \displaystyle \sum_{c_1 + \cdots + c_{N - 1} = M - c} \frac{(c_1 + 1) \cdots (c_{N - 1} + 1)(c + 1)}{\Lambda (N, M)} =
    \frac{(c + 1) \,\Lambda (N - 1, M - c)}{\Lambda (N, M)}. \end{array} $$
 This, together with Lemma~\ref{lem:total-mass}, implies that
 $$ \lim_{t \to \infty} P (\imex_t (x) = c) = (c + 1) {M - c + 2N - 3 \choose 2N - 3} \bigg/ {M + 2N - 1 \choose 2N - 1}. $$
 This proves the first part of Theorem~\ref{th:exchange}.
 Now, observe that
 $$ \begin{array}{l}
    \displaystyle (c + 1) {M - c + 2N - 3 \choose 2N - 3} \bigg/ {M + 2N - 1 \choose 2N - 1} \vspace*{10pt} \\ \hspace*{20pt} =
    \displaystyle (c + 1) \ \frac{(2N - 1)(2N - 2)}{(M + 2N - 1)(M + 2N - 2)} \ \frac{(M - c + 2N - 3) \cdots (M - c + 1)}{(M + 2N - 3) \cdots (M + 1)}. \end{array} $$
 In particular, when~$N$ and~$T$ are large, this is approximately
 $$ (c + 1) \bigg(\frac{2N}{M} \bigg)^2 \bigg(1 - \frac{c}{M} \bigg)^{2N} \approx \frac{4c}{T^2} \ \bigg(1 - \frac{c}{NT} \bigg)^{2N} \approx \frac{4c}{T^2} \ e^{- 2c/T}. $$
 This completes the proof of Theorem~\ref{th:exchange}.
 The proof of Theorem~\ref{th:saving} is exactly the same since both models converge to the same stationary distribution~$\pi = \pi_{\imex} = \pi_{\save}$.
\end{proofof}


\end{document}